\documentclass[11pt]{amsart}
\usepackage{graphicx} 
\usepackage{mathpazo}
\usepackage[margin=2cm]{geometry}
\usepackage{hyperref}

\theoremstyle{plain}
\newtheorem{theorem}{Theorem}

\newtheorem{lemma}[theorem]{Lemma}

\newtheorem{corollary}[theorem]{Corollary}
\theoremstyle{remark}
\newtheorem{remark}[theorem]{Remark}

\makeatletter
\def\@bignumber#1#2{%
  \ifx#2\end
    #1\let\next\@gobble
  \else
    #1\hspace{0pt plus 1pt}\let\next\@bignumber
  \fi
  \next#2}
\newcommand{\bignumber}[1]{\@bignumber#1\end}
\makeatother

\title{Factorizations into irreducible numerical semigroups}
\author{Pedro A. García-Sánchez}
\date{October 2024}

\begin{document}

\begin{abstract}
Every numerical semigroup can be expressed as an intersection of irreducible numerical semigroups. We show that the unions of sets of lengths of factorizations of numerical semigroups into irreducible numerical semigroups are all equal to $\mathbb{N}_{\ge 2}$.  
\end{abstract}

\maketitle

\section{Introduction}

A \emph{numerical semigroup} is a co-finite submonoid of the set of non-negative integers (hereafter denoted as $\mathbb{N}$) under addition.

A numerical semigroup $S$ is said to be \emph{irreducible} if it cannot be expressed as the intersection of two numerical semigroups properly containing it. Every numerical semigroup is the intersection of finitely many irreducible numerical semigroups \cite[Proposition~4.4]{ns}. Given $S$ a numerical semigroup and $S_1,\dots,S_n$ irreducible numerical semigroups such that $S=S_1\cap \dots \cap S_n$, we will say that $S_1\cap \dots \cap S_n$ is a \emph{factorization} of $S$ (of length $n$) into the irreducibles $S_1,\dots,S_n$ if $S\neq \bigcap_{j\in J}S_j$ for every proper subset $J$ of $\{1,\dots,n\}$. This means that $S$ can be expressed as an intersection of $S_1,\dots,S_n$, and that this intersection cannot be refined. This definition of factorization follows the spirit of minimal idempotent factorization as introduced in \cite{c-t} (minimal factorizations appear already in \cite{a-t}), if we consider our monoid to be the set of all numerical semigroups, which we denote by $\mathfrak{N}$, and the operation is the intersection of numerical semigroups.

In \cite{d-gs-r}, the question whether there is a positive integer $N$ such that every numerical semigroup admits a factorization into at most $N$ irreducibles was raised. The answer to this question was answered in the negative in \cite{b-f}, where two families of numerical semigroups were presented with unbounded factorizations into irreducibles. 

In the present manuscript, we give yet another family that shows that the unions of sets of lengths of factorizations are always the set of integers greater than one.

Let us recall the (classic) concept of unions of sets of lengths of factorization in an atomic monoid $M$. For every positive integer $k$, the set $\mathcal U_k (M)$ is defined as the set of all $l\in \mathbb{N}$ for which there is an equation of the form
\[
u_1 \dots u_k = v_1 \dots v_{l},
\]
where $u_1, \dots , u_k, v_1, \dots, v_{l}$ are irreducible elements of $M$. In other words, the sets $\mathcal U_k (M)$ are the unions of all sets of lengths containing $k$. Research focused on their inner structure as well as on their supremum $\rho_k (M) = \sup \mathcal U_k (M)$, called the \emph{$k^{th}$ elasticity}, can be found in \cite{Ga-Ge09b, F-G-K-T17, Bl-Ga-Ge11a, Ba-Sm18, Tr19a, Sm13a, Fa-Tr18a, Ge-Go-Tr21, Ge-Kh22b, Po20a} for commutative integral domains, for unit-cancellative commutative monoids (in particular for numerical monoids and Krull monoids), and also for non-commutative rings and monoids. 


Chapter~3 of \cite{ns} contains an extensive introduction to irreducible numerical semigroups, and provides several characterizations on when a numerical semigroup is decomposed as the intersection of irreducible numerical semigroups. In particular, we will make use of \cite[Proposition~4.48]{ns}, where a characterization is given in terms of the special gaps of the numerical semigroup.

The set of \emph{special gaps} of a numerical semigroup $S$ is the set of gaps $g$ of $S$ such that $S\cup\{g\}$ is a numerical semigroup. It follows easily that 
\[ 
\operatorname{SG}(S)=\{g\in \mathbb{N}\setminus S : g\in \operatorname{PF}(S), 2g\in S(i)\},
\]
where $\operatorname{PF}(S)$ is the set of \emph{pseudo-Frobenius} numbers of $S$, that is, the set of integers $z$ such that $z+s\in S$ for all $s\in S\setminus \{0\}$.

Proposition~4.48 in \cite{ns} states that $S=S_1\cap \dots \cap S_n$, with $S_1,\dots, S_n$ numerical semigroups containing $S$, if and only if for all $h\in \operatorname{SG}(S)$, there is $i\in \{1,\dots,n\}$ such that $h\not\in S_i$.

Notice, that in particular \cite[Proposition~4.8]{ns}, implies that the maximal length of a factorization of a numerical semigroup $S$ into irreducibles is at most the cardinality of $\operatorname{SG}(S)$.

\section{Unions of sets of lengths of factorizations}\label{sec-unions}


Given a numerical semigroup $S$, we denote by $\mathsf{L}^m(S)$ the set of positive integers $n$ such that there exists a factorization of $S$ into $n$ irreducibles, and it is known as the \emph{set of lengths} of factorizations of $S$ (we are using the superscript $m$ to emphasize that we are considering minimal factorization in the sense that they cannot be refined; see for instance \cite{g} for a nice overview on the sets of lengths of factorizations in monoids). 

The elasticity of the factorizations of $S$ is defined as $\rho^m(S)=\sup(\mathsf{L}^m(S))/\min(\mathsf{L}^m(S))$; one can define the elasticity of $\mathfrak{N}$ as $\rho^m(\mathfrak{N})=\sup\{ \rho^m(S) : S\in\mathfrak{N}\}$. Lower and upper bounds for $\min(\mathsf{L}^m(S))$ can be found in \cite{r-b}.

For a positive integer $k$, the union of sets of lengths $\mathcal{U}_k^m(\mathfrak{N})$ is defined as the set of positive integers $l$ such that there exists a numerical semigroup having a factorization of length $k$ and another of length $l$. This concept was introduced in \cite{ch-s}, and since then it has been studied in many different contexts and scenarios as mentioned in the introduction. The study of $\mathcal{U}_k^m(\mathfrak{N})$ requires special care, since it is probably the first time that unions of sets of lengths have been studied for minimal idempotent factorizations.

Our main goal, is to prove that $\mathcal{U}^m_k(\mathfrak{N})=\mathbb{N}_{\ge 2}$ (the set of integers greater than one) for every integer $k\ge 2$. To do this, we relay on the set of factorizations of a very particular family of numerical semigroups, which we introduce next.

Let $i$ be an odd integer greater than three. Define 
\[
T(i)=\{0,(i+1)/2, \dots, i-1,i+1,\to\},
\]
which is  an irreducible numerical semigroup by \cite[Proposition~2.7]{b-r}.

Let $S(i)=\langle 2,i\rangle \cap T(i)$. Then,
\begin{equation}\label{eq:Si}
S(i)=\begin{cases}
\left\{0, \frac{i+1}2, \frac{i+1}2+2,\dots ,i-1,i+1,\to\right\}, \text{ if }  (i+1)/2 \text{ is even},\\
\left\{0, \frac{i+1}2+1, \frac{i+1}2+3,\dots ,i-1,i+1,\to\right\}, \text{ otherwise. }
\end{cases}
\end{equation}

We start by studying the special gaps of $S(i)$.

\begin{lemma}\label{lem:sg-si}
Let $i$ be an odd positive integer. 
\begin{itemize}
    \item If $(i+1)/2$ is even, then \[\operatorname{SG}(S(i))=([(i+1)/4,(i-1)/2]\cap 2\mathbb{N})\cup \{(i+1)/2+2k+1 : k \in \{0,\dots,(i-3)/4\}\}.\]
    \item If $(i+1)/2$ is odd, then
    \[
    \operatorname{SG}(S(i))=([(i+3)/4,(i-1)/2]\cap 2\mathbb{N})\cup \{(i+1)/2+2k : k \in \{0,\dots,(i-1)/4\}\}.
    \]
\end{itemize}
\end{lemma}
\begin{proof}
We distinguish two cases, depending on the parity of $(i+1)/2$.

Suppose that $(i+1)/2$ is even. By \eqref{eq:Si}, if $g$ is a gap larger than $(i+1)/2$, then $g=(i+1)/2+2k+1$ for some non-negative integer $k$. For every non-negative integer $l$, $g+(i+1)/2+2l= i+1+2(l+k)+1\in S$. Also, $g+n\in S(i)$ if $n$ is an integer greater than $i$, and $2g\in S(i)$. Therefore, $g\in \operatorname{SG}(S(i))$. 

Now take $g$ a positive integer such that $g< (i+1)/2$; whence $g\le (i+1)/2-1=(i-1)/2$. If $g$ is odd, then $g+(i+1)/2$ is odd and smaller than or equal to $i$. Thus, $g+(i+1)/2$ is not in $S$, which means that $g\not\in \operatorname{SG}(S(i))$. If $g$ is even, then $g+s$, with $s\in S(i)\setminus\{0\}$ is either larger than $i$ or it is even and smaller than $i-1$; in both cases $g+s\in S(i)$. The condition $2g\in S(i)$, is satisfied if $g$ is larger than or equal to $(i+1)/4$.

The proof for the case $(i+1)/2$ is odd is similar and it is left to the reader.
%
\end{proof}
Set 
\[
B(i)=\begin{cases}
\left\{\frac{i+1}2+1, \frac{i+1}2+3,\dots ,i\right\}, \text{ if }  (i+1)/2 \text{ is even},\\
\left\{\frac{i+1}2, \frac{i+1}2+2,\dots ,i\right\}, \text{ otherwise. }
\end{cases}
\]

\begin{lemma}\label{lem:fact-si-full}
    Let $i$ be an odd integer. Then, $S(i)=\bigcap_{j\in B(i)} T(j)$ is a factorization of $S(i)$.
\end{lemma}
\begin{proof}
Suppose that $(i+1)/2$ is even and let $j\in B(i)$. Then, $j=(i+1)/2+2k+1$ for some $k\in \{0,\dots,(i-3)/4\}$. It follows that 
\[
T(j)=\{0,(i+1)/4+k+1,\dots, (i+1)/2+2k,(i+1)/2+2k+2,\to\}.
\]

Therefore,
\[
\{h\in \operatorname{SG}(i) : g\not\in T(j)\}=([(i+1)/4,(i+1)/4+k]\cap 2\mathbb{N}) \cup\{j\}.
\] 

By \cite[Proposition~4.48]{ns}, we have that $S(i)=\bigcap_{j\in B(i)} T(j)$. Moreover, this intersection cannot be refined.

The case $(i+1)/2$ odd is analogous.
\end{proof}

If $i$ is an odd integer such that $(i+1)/2=2k$, with $k$ a positive integer, then the proof of Lemma~\ref{lem:fact-si-full} is telling us that $S(i)$ admits a factorization of length $1+(i-3)/k=(i+2)/4=k$. If $k\ge 3$, this implies that $S(i)$ has a factorization of length two and another of length $k$. This, in particular, forces $\mathcal{U}_2^m(\mathfrak{N})= \mathbb{N}_{\ge 2}$. 


\begin{corollary}
    The sequence $\{\rho^m(S(i))\}_{i\in 4\mathbb{N}+3}$ is not bounded. In particular, $\rho^m(\mathfrak{N})=\infty$.
\end{corollary}
\begin{proof}
    Let $k$ be a non-negative integer, and let $i=4k+3$. Then, $(i+1)/2=2(k+1)$ is even. Clearly, $\min (\mathsf{L}^m(S(i)))=2$. By Lemma~\ref{lem:fact-si-full}, $(k+1)\le \mathsf{L}^m(S(i))$. Thus, $\rho^m(S(i))\ge (k+1)/2$, and the statement follows.
\end{proof}

In order to prove that $\mathcal{U}_k^m(\mathfrak{N})=\mathbb{N}_{\ge 2}$ for every integer $k$ larger than two, we will make use of the following alternate factorizations of $S(i)$.

\begin{lemma}\label{lem:factorization-si-general}
    Let $i$ be an odd integer such that $(i+1)/2$ is even. Then, for all $t\in \{1,\dots, (i-3)/4\}$, 
    \begin{equation}\label{eq:factorization-si-general}
            S(i)=\langle 2, (i+1)/2+2t+1\rangle \cap \left( \bigcap_{k=t}^{(i-3)/4} T\left(\frac{i+1}2+2k+1\right) \right).
    \end{equation}
    Moreover, this is a factorization of $S(i)$ (that is, the intersection cannot be refined).
\end{lemma}
\begin{proof}
    We are going to make use of \cite[Proposition~4.48]{ns} once more: we show for every $g\in \operatorname{SG}(S(i))$, there is exactly one numerical semigroup not containing $g$ in the right-hand side of \eqref{eq:factorization-si-general}. As $i=\frac{i+1}2+2\frac{i-3}4+1$, we deduce from the proof of Lemma~\ref{lem:fact-si-full}, that 
    \begin{align*}
            \{h\in \operatorname{SG}(i) : g\not\in T(i)\}& =([(i+1)/4,(i+1)/4+(i-3)/4]\cap 2\mathbb{N}) \cup\{i\}\\
            &=([(i+1)/4,(i-2)/2]\cap 2\mathbb{N}) \cup\{i\}.
    \end{align*}   
    This means by Lemma~\ref{lem:sg-si} that every special gap of $S(i)$ smaller than $(i+1)/2$ is not  in $T(i)$.

    Notice that every special gap $g$ of $S(i)$ larger than $(i+1)/2$ and smaller than $(i+1)/2+2t+1$ is odd, and thus $g\not\in \langle 2,(i+1)/2+2t+1\rangle$.

    Finally, if $g$ is a special gap of $S(i)$ in the interval $[(i+1)/2+2t+1,i]$, then $g=(i+1)/2+2k+1$ for some $k\in \{t,\dots,(i-3)/4\}$, and so $g\not \in T((i+1)/2+2k+1)$.
\end{proof}

With this result we have shown that if $i$ is an odd integer such that $(i+1)/2=2k$, with $k$ an integer greater than one, then $S(i)$ admits factorizations of lengths ranging from two to $k$, thus proving the following result.




\begin{theorem}
    Let $k$ be an integer greater than two. Then, $\mathcal{U}_k^m(\mathfrak{N})=\mathbb{N}_{\ge 2}$.
\end{theorem}

\begin{remark}
    Notice that Lemma~\ref{lem:fact-si-full} describes a factorization of length $k$ for $S(4k-1)$, and 
    Lemma~\ref{lem:factorization-si-general} offers  another factorization of length $k$, if we take $t=1$ and $k$ larger than two. This means that $S(4k-1)$ admits at least two different factorizations of the same length, and thus $(\mathfrak{N},\cap)$ is not length-factorial (see \cite{l-f}).
\end{remark}

\section{Future work}

Most of the experiments that led to our results were performed with the help of the \texttt{GAP} \cite{gap} package \texttt{numericalsgps} \cite{numericalsgps}. It seems that the sets of lengths of factorizations of the semigroups $S(i)$ are intervals. As a matter of fact, we were not able to find a numerical semigroup whose sets of lengths of factorizations is not an interval. It would be interesting to prove if for any numerical semigroup, its set of lengths of factorizations is an interval, or at least find a precise characterization of these sets of lengths.

\section*{Acknowledgements}

The author wishes to thank the Doctoral Academy of the University of Graz for their hospitality during the last visit of the author to that university. He also wishes to thank Alfred Geroldinger for proposing the problem central to this manuscript and for the fruitful discussions that followed after.

The author is partially supported by the grant number ProyExcel\_00868 (Proyecto de Excelencia de la Junta de Andalucía) and by the Junta de Andaluc\'ia Grant Number FQM--343. He also acknowledges financial support from the grant PID2022-138906NB-C21 funded by MICIU/AEI/10.13039/\bignumber{501100011033} and by ERDF ``A way of making Europe'', and from the Spanish Ministry of Science and Innovation (MICINN), through the ``Severo Ochoa and María de Maeztu Programme for Centres and Unities of Excellence'' (CEX2020-001105-M).

\end{document}